\documentclass[12pt]{article}
\usepackage{e-jc}
\usepackage{amsmath,amssymb,amsthm}
\usepackage{epsfig,graphicx,verbatim}
\numberwithin{equation}{section}
\newtheorem{thm}[equation]{Theorem}
\newtheorem{prop}[equation]{Proposition}
\newtheorem{cor}[equation]{Corollary}

\theoremstyle{definition}
\newtheorem{remark}[equation]{Remark}
\newcounter{mycount}

\newenvironment{letlist}{\begin{list}{(\alph{mycount})}%
   {\usecounter{mycount}\labelwidth=1cm\itemsep 0pt}}{\end{list}}

\def\HH{{\mathbb H}}
\def\TT{{\mathbb T}}
\def\sF{{\mathcal E}}
\def\cE{\sF}
\def\cF{{\mathcal F}}

\def\sC{{\mathcal C}}

\def\ZRC{Z^{\text{RC}}}
\def\s{\sigma}
\def\ZI{Z^{\text{I}}}
\def\qq{\qquad}
\def\q{\quad}

\def\b{\beta}
\def\d{\delta}

\def\g{\gamma}

\def\rc{random-cluster}
\def\ZZ{{\mathbb Z}}
\def\RR{{\mathbb R}}

\def\Pr{{\mathbb P}}

\def\la{\langle}
\def\ra{\rangle}
\def\Si{\Sigma}
\def\Om{\Omega}
\def\om{\omega}

\def\oo{\infty}
\def\deg{\text{deg}}

\def\pc{p_{\mathrm c}}
\def\bc{{\b_{\mathrm c}}}
\def\bsd{\b_{\mathrm {sd}}}
\def\rad{{\mathrm{rad}}}
\def\frW{\phi_{r,2}^W}
\def\Gw{G^{\mathrm w}}
\def\ol#1{\overline{#1}}

\def\be{\begin{equation}}
\def\ee{\end{equation}}
\def\sm{\setminus}
\def\resp{respectively}

\def\ZI{Z^{\text{I}}}

\def\ZE{Z^{\text{E}}}
\def\rd{\text{\rm d}}
\def\Gd{G_\rd}
\def\Vd{V_\rd}
\def\Ed{{E_\rd}}

\def\ed{e_\rd}

\def\pd{\partial}
\def\lra{\leftrightarrow}
\def\es{\varnothing}
\def\symdif{\,\bigtriangleup\,}

\def\fpq{\phi_{p,2}}
\def\fbpq{\phi_{\br,q}}

\def\bp{\mathbf{p}}

\def\br{{\mathbf r}}

\newcommand\bbzz{{\mathbb Z}_2}
\newcommand\set[1]{\{#1\}}
\newcommand\pei{\pi_{E_1}}
\newcommand\pee{\pi_{E_2E_1}}
\newcommand\oi{\set{0,1}}
\newcommand\tF{\widetilde F}
\newcommand\Be{\operatorname{Be}}
\newcommand\ax{a${}'$}
\newcommand\bx{b${}'$}

\date{\dateline{8 October 2008}{26 February 2009}}
\title{Random even graphs}

\author{Geoffrey Grimmett\\
\small Statistical Laboratory, Centre for Mathematical Sciences,\\
\small  University of Cambridge,
Wilberforce Road, Cambridge CB3 0WB, U.K.\\
\small \texttt{g.r.grimmett@statslab.cam.ac.uk}\\
\small \texttt{http://www.statslab.cam.ac.uk/$\sim$grg/}\\
\and Svante Janson\\
\small Department of Mathematics, Uppsala University,\\
\small  PO Box 480,
SE-751~06 Uppsala, Sweden\\
\small \texttt{svante@math.uu.se}\\
\small \texttt{http://www.math.uu.se/$\sim$svante/}}

\begin{document}
\maketitle

\begin{abstract}
We study a random even subgraph of a finite graph $G$ with a general
edge-weight $p\in(0,1)$. We demonstrate how it may be
obtained from a certain random-cluster measure on $G$, and
we propose a sampling algorithm
based on coupling from the past.
A random even subgraph of a planar lattice undergoes a phase transition
at the parameter-value $\frac 12 \pc$, where $\pc$ is the critical point
of the $q=2$ \rc\ model on the dual lattice. The properties of such a graph 
are discussed, and are related to Schramm--L\"owner
evolutions (SLE). 
\end{abstract}

\section{Introduction}\label{sec:intro}
Our purpose in this paper is to study a random even subgraph
of a finite graph $G=(V,E)$, and to show how to sample such a subgraph.
A subset $F$ of $E$ is called \emph{even}
if, for all $x \in V$, $x$ is incident to an even number of elements
of $F$. We call the subgraph $(V,F)$ \emph{even} if $F$ is even, and
we write 
$\sF$ for the set of all even subsets $F$ of $E$. It is
standard that every even set $F$ may be decomposed as an edge-disjoint
union of cycles. Let $p \in [0,1)$.  The \emph{random even subgraph}
of $G$ with parameter $p$ is that with law
\begin{equation}
  \label{eq:res}
\rho_p(F) = \frac 1 \ZE
p^{|F|} (1-p)^{|E\sm F|}, \qquad F \in \sF,
\end{equation} 
where
$\ZE=\ZE_G(p)$ is the appropriate normalizing constant.

We may express $\rho_p$ as follows in terms of product measure on $E$. Let
$\phi_p$ be product measure with density $p$ on the configuration space $\Om=\{0,1\}^E$.
For $\om\in\Om$ and $e \in E$, we call $e$ {\it$\om$-open\/} if $\om(e)=1$,
and {\it$\om$-closed\/} otherwise. Let $\pd\om$ denote the set of vertices $x\in V$ that are
incident to an odd number of $\om$-open edges. Then
\begin{equation}
\rho_p(F) = \frac{\phi_p(\om_F)}{\phi_p(\pd\om=\es)}, \qq F\in \sF,
\label{new1}
\end{equation}
where $\om_F$ is the edge-configuration whose open edge-set is $F$.
In other words, 
$\phi_p$ describes the  
random subgraph of $G$ obtained by randomly and independently deleting
each edge with probability $1-p$, and $\rho_p$ is the law of this random
subgraph conditioned on being even.

Random even graphs are closely related to the Ising model and the random-cluster
model on $G$, and we review these models briefly.
Let $\b\in(0,\oo)$  and
\begin{equation}
p=1-e^{- 2\beta} = \frac{2\tanh\b}{1+\tanh\b}. 
\label{8.pj}
\end{equation}
The Ising model on $G$ has configuration space
$\Sigma = \{-1,+1\}^V$, and probability measure
\begin{equation}
\pi_{\beta} (\s) = \frac 1{\ZI}
\exp\biggl\{\beta\sum_{e \in E} \s_x\s_y\biggr\},
\qq \s\in \Sigma,
\label{8.Isingdef}
\end{equation}
where $\ZI=\ZI_G(\b)$ is the partition function that makes $\pi_\b$ a probability measure,
and $e=\la x,y\ra$ denotes an edge with endpoints $x$, $y$.
A \emph{spin-cluster} of a configuration $\s\in\Si$ is a maximal connected
subgraph of $G$ each of whose vertices $v$ has the same spin-value
$\s_v$. A spin-cluster is termed a $k$ cluster if $\s_v=k$ for all $v$ 
belonging to the cluster.
An important quantity associated with the Ising model is
the `two-point correlation function' 
\begin{equation}
\tau_{\beta}(x,y) = \pi_{\beta}(\s_x=\s_y)-\tfrac 12
=\tfrac12\pi_{\b}(\s_x\s_y),
\qq x,y \in V,
\label{8.taudef}
\end{equation}
where $P(f)$ denotes the expectation of a random variable $f$ under
the probability measure $P$.

The \rc\ measure on $G$ with parameters $p\in(0,1)$ and $q=2$ is given as follows [it
may be defined for general $q>0$ but we are concerned here only with the case $q=2$]. 
Let
\begin{align}
\fpq(\om) &= \frac1{\ZRC}
\biggl\{ \prod_{e\in E} p^{\om(e)} (1-p)^{1-\om(e)}
                           \biggr\} 2^{k(\om)}\nonumber\\
&= \frac 1 \ZRC p^{|\eta(\om)|} (1-p)^{|E\sm\eta(\om)|} 2^{k(\om)},
 \qq \om \in \Om,
\label{old2.53}
\end{align}
where $k(\om)$ denotes the number of $\om$-open
components on the vertex-set $V$, $\eta(\om)=
\{e\in E: \om(e)=1\}$ is the set of open edges, and
$\ZRC=\ZRC_G(p)$ is the appropriate normalizing factor.

The relationship between the Ising and \rc\ models on $G$ is well
established, and hinges on the fact that, in the notation introduced
above, 
$$
\tau_{\beta}(x,y) = \tfrac12 \fpq(x \lra y),
$$
where $\{x\lra y\}$ is the event that $x$ and $y$ are
connected by an open path. See \cite{G-RC} for an account of the
\rc\ model. There is a relationship between the Ising model
and the random even graph also, known misleadingly as the
`high-temperature expansion'. This may be stated as follows. For completeness,
we include a proof of this standard fact at
the end of the section, see also \cite{Bax}.

\begin{thm} \label{Tising}
Let $2p=1-e^{-2\b}$ where $p\in(0,\frac12)$, 
and consider the Ising model with
inverse temperature $\b$. 
Then
$$
\pi_{\b,2}(\s_x\s_y) 
=  \frac{\phi_p(\pd \om = \{x,y\})}{\phi_p(\pd \om = \es)},
\qquad x,y\in V, \ x\ne y.
$$
\end{thm}

A corresponding conclusion is valid for the product of $\s_{x_i}$ over any
even family of distinct $x_i\in V$.

This note is laid out in the following way. In Section \ref{sec:ures} we define a random even
subgraph of a finite or infinite graph, and we explain how to
sample a \emph{uniform} even subgraph. In Section \ref{sec:rcre} we
explain how to sample 
a non-uniform random even graph, starting with a sample from a
random-cluster measure. 
An algorithm for exact sampling is presented in Section \ref{sec:sampling} based
on the method of coupling from the past.
The structure of random even subgraphs of the square and hexagonal lattices
is summarized in Section \ref{sec:revenL}. 

In a second paper \cite{even2}, we study the asymptotic properties of
a random even subgraph of the complete graph $K_n$. Whereas the
special relationship with 
the \rc\ and Ising models is the main feature of the current work, the
analysis of \cite{even2} is more analytic, and extends to random
graphs whose vertex degrees
are constrained to lie in any given subsequence of the non-negative integers.

\begin{remark}\label{Rpe}
The definition \eqref{eq:res} may be generalized by replacing the 
  single parameter $p$ by a family $\bp = (p_e: e\in E)$, just as
  sometimes is done for the \rc\ measure \eqref{old2.53}, see for example
  \cite{Sok05}; we let
\begin{equation}\label{rpe}
	\rho_{\bp}(F)=\frac1 Z \prod_{e\in F} p_e\prod_{e\notin F} (1-p_e).
  \end{equation}
For simplicity we will mostly consider the case of a single $p$.
\end{remark}

\begin{proof}[Proof of Theorem \ref{Tising}]
For $\s\in\Si$, $\om\in\Om$, let
\begin{align}
Z_p(\s,\om) 
&= \prod_{e=\la v,w\ra} \Bigl\{(1-p)\d_{\om(e),0} + p\s_v\s_w\d_{\om(e),1}\Bigr\}\nonumber\\
&= p^{|\eta(\om)|}(1-p)^{|E\sm\eta(\om)|} \prod_{v\in V} \s_v^{\deg(v,\om)},
\label{eq:0}
\end{align}
where $\deg(v,\om)$ is the degree of $v$ in the `open' graph
$(V,\eta(\om))$. Then 
\begin{align}
\sum_{\om\in\Om} Z_p(\s,\om) 
&=
\prod_{e=\la v,w\ra} (1-p+p\s_v\s_w)
=
\prod_{e=\la v,w\ra} e^{\b(\s_v\s_w-1)}\nonumber\\
&= e^{-\b|E|} \exp\left(\b\sum_{e=\la v,w\ra} \s_v\s_w\right),
\qquad \s\in\Si.
\label{eq:1}
\end{align}
Similarly, 
\begin{align}
\sum_{\s\in\Si} Z_p(\s,\om) &=
2^{|V|} p^{|\eta(\om)|}(1-p)^{|E\sm\eta(\om)|} 1_{\{{\pd\om =\es}\}},
\qquad \om\in\Om,
\label{eq:2a}
\intertext{and}
\sum_{\s\in\Si} \s_x\s_yZ_p(\s,\om) &=
2^{|V|} p^{|\eta(\om)|}(1-p)^{|E\sm\eta(\om)|} 1_{\{{\pd\om =\{x,y\}}\}},
\qquad \om\in\Om.
\label{eq:2b}
\end{align}

By \eqref{eq:1},
$$
\pi_{\b,2}(\s_x\s_y) = 
\frac {\sum_{\s,\om}\s_x\s_y Z_p(\s,\om)}
{\sum_{\s,\om} Z_p(\s,\om)},
$$
and the claim follows by \eqref{eq:2a}--\eqref{eq:2b}.
\end{proof}

\section{Uniform random even subgraphs}\label{sec:ures}
\subsection{Finite graphs}\label{ssec:fin}

In the case $p=\frac12$ in \eqref{eq:res}, every even subgraph has
the same probability, so $\rho_{\frac12}$ describes a uniform random even
subgraph of $G$. Such a random subgraph can be obtained as follows.

We identify the family of all spanning subgraphs of $G=(V,E)$ 
with the family $2^E$ of all subsets of $E$. This family
can further be identified
with $\set{0,1}^E=\bbzz^E$, and is thus a vector space over $\bbzz$;
the addition is componentwise addition modulo 2 in $\set{0,1}^E$,
which translates into taking the symmetric difference of edge-sets:
$F_1+F_2=F_1\symdif F_2$ for $F_1,F_2\subseteq E$.

The family of even subgraphs of $G$ forms a subspace $\sF$ of this
vector space $\set{0,1}^E$, since $F_1+F_2=F_1\symdif F_2$ is even if
$F_1$ and $F_2$ are even. (In fact, $\sF$ is the cycle space $Z_1$ in
the $\bbzz$-homology of $G$ as a simplicial complex.)
In particular, the number of even subgraphs of $G$ equals $2^{c(G)}$
where $c(G)=\dim(\sF)$;
$c(G)$ is thus the number of independent cycles in $G$, and, as is
well known,
\begin{equation}\label{c}
c(G)=|E|-|V|+k(G).  
\end{equation}

\begin{prop}\label{P1}
  Let $C_1,\dots,C_c$ be 
a maximal set of independent cycles in $G$.
Let $\xi_1,\dots,\xi_c$ be independent $\Be(\frac12)$ random variables
(i.e., the results of fair coin tosses).
Then $\sum_{i}\xi_i C_i$ is a uniform random even subgraph of $G$.
\end{prop}

\begin{proof}
  $C_1,\dots,C_c$ is a basis of the vector space $\sF$ over $\bbzz$.
\end{proof}

One standard way of choosing 
  $C_1,\dots,C_c$ is exploited in the next proposition. Another, for
  planar graphs, is given by the boundaries of the finite faces;
this will be used in Section \ref{sec:revenL}. In the following proposition, we
use the term \emph{spanning subforest} 
of $G$ to mean a maximal forest of $G$, that is, the union of
a spanning tree from each component of $G$.

  \begin{prop}\label{P2}
Let $(V,F)$ be a spanning subforest of $G$.
Each subset $X$ of $E\setminus F$ can be completed by a unique
$Y\subseteq F$ to an even edge-set $E_X = X\cup Y\in\sF$.
Choosing a uniform random subset $X\subseteq E\setminus F$ thus gives
a uniform random even subgraph $E_X$ of $G$.
  \end{prop}

  \begin{proof}
It is easy to see, and well known, that each edge $e_i\in E\setminus
F$ can be completed by edges in $F$ to a unique cycle $C_i$; these
cycles form a basis of $\sF$ and the result follows by Proposition
\ref{P1}.
(It is also easy to give a direct proof.)	
  \end{proof}

\subsection{Infinite graphs}\label{ssec:infin}

Here, and only here, we consider even subgraphs of \emph{infinite} graphs.
Let $G=(V,E)$ be a locally finite, infinite graph.
We call a set $\cF\subset 2^E$ 
\emph{finitary} if each edge in $E$ belongs to only a
finite number of elements in $\cF$. If $G$ is countable (for example, if $G$ is locally finite and
connected), then any finitary $\cF$ is necessarily countable.
If $\cF\subset 2^E$ is finitary, then the (generally
infinite) sum $\sum_{x\in\cF}x$ is a well-defined
element of $2^E$, by considering one coordinate (edge) at
a time; if, for simplicity,
$\cF=\set{x_i: i\in I}$, then $\sum_{i\in I} x_i$ includes a given edge
$e$ if and only if $e$ lies in an odd number of the $x_i$.

We can define the even subspace $\sF$ of $2^E$ as before. (Note that we
need $G$ to be locally finite in order to do so.) 
If $\cF$ is a finitary subset of $\cE$,
then $\sum_{x\in \cF}x\in\cE$.

A \emph{finitary basis} of $\cE$ is a finitary
subset $\cF\subset\cE$ such that every element of
$\cE$ is the sum of a unique subset
$\cF'\subseteq\cF$; in other words, if the linear (over
$\ZZ_2$) map $2^{\cF}\to\cE$ defined by
summation is an isomorphism. 
(A finitary basis is not a vector-space basis in the usual algebraic
sense since the summations are generally infinite.) 

We define an \emph{infinite cycle} in $G$ to be a subgraph
isomorphic to $\ZZ$, i.e., a doubly infinite path. (It is
natural to regard such a path as a cycle passing through infinity.)
Note that, if $F$ is an even subgraph of $G$, then every
edge $e \in F$ belongs to some finite or infinite cycle in
$F$: if no finite cycle contains $e$, removal of
$e$ would disconnect the component of $F$ that contains
$e$ into two parts; since $F$ is even both parts have to
be infinite, so there exist infinite rays from the endpoints of
$e$, which together with $e$ form an infinite cycle.

\begin{prop}
The space $\cE$ has a finitary basis. We may choose such
a finitary basis containing only finite or infinite cycles.
\end{prop}

\begin{proof}
  It suffices to consider the case when $G$ is connected, and hence countable.
We construct a finitary basis by  induction.
Order the edges in a fixed but arbitrary way as $e_1 < e_2 <\cdots$. 
Let $h_1$ be the first edge that belongs to
an even subgraph of $G$, and
  choose a (finite or infinite) cycle $C_1$ containing $h_1$.
Having chosen $h_1$, $C_1$, \dots, $h_n$,
  $C_n$, consider the subspace $\cE_n$ of all even subgraphs of $G$ 
containing none of $h_1,\dots,h_n$. 
If $\cE_n=\{\es\}$, we stop, and write $\cF=\{C_1,C_2,\dots,C_n\}$. 
Otherwise, let $h_{n+1}$ be the earliest edge
belonging to some non-trivial even subgraph $F_n \in\cE_n$, and choose
a cycle $C_{n+1}\subset F_n$ containing $h_{n+1}$. Either this process stops
after finitely many steps, with the cycle set $\cF$, or it continues forever,
and we write $\cF$ for the countable set of cycles thus obtained.
Finally, write $H=\{h_1,h_2,\dots\}$. We shall assume that $H \ne\es$,
since the proposition is trivial otherwise.

We claim that $\cF$ is a finitary basis for $\cE$. Note that 
\be
h_{n}\in C_{n},\qq h_j\notin C_{n}\text{  for }j < n.
\label{gen7}
\ee
Let $e \in E$, say $e=e_r$.
If $e_r =h_s$ for some $s$, then  $e_r$ lies in only finitely many
of the $C_j$. If $e_r\in E\sm H$ and $h_s < e_r < h_{s+1}$ for
some $s$ (or $h_s < e_r$ for all $s$), then $e_r$ lies in no member of $\cE_s$, so that it
lies in only finitely many of the $C_j$. If $e_r < h_1$, then $e_r$ lies
in no $C_j$.  In conclusion, $\cF$ is
finitary.

Next we show that no element $F \in \cE$ has more than one representation
in terms of $\cF$. Suppose, on the contrary, that $\sum_i \xi_i C_i =
\sum_i \psi_i C_i$. Then the sum of these two summations is the
empty set. By \eqref{gen7}, there is no
non-trivial linear combination of the $C_i$ that equals the empty set,
and therefore $\xi_i=\psi_i$ for every $i$.

Finally, we show that $\cF$ spans $\cE$, which is to say that
the map $2^\cF \to \cE$ defined by summation has range $\cE$. 
Let $\ol \cF$ be the subspace of
$\cE$ spanned by $\cF$.
For $H'\subseteq H$, there is a unique element $F' \in \ol\cF$ such that 
$F' \cap H = H'$; $F'$ is obtained by an inductive construction that considers
the $C_j$ in order of increasing $j$, and includes a given $C_j$ if: either $h_j\in H'$
and $h_j$ lies in an even number of the $C_i$ already included, or
$h_j \notin H'$ and $h_j$ lies in an odd number of the
$C_i$ already included. 

Let $F \in \cE$. 
By the above, there is a unique element $F'\in \ol\cF$ 
satisfying $F' \cap H = F \cap H$. Thus, $F + F'$ is
an even subgraph having empty intersection with $H$. 
Let $e_r$ be the earliest edge in $F+F'$, if such an edge exists.
Since $e_r \in F + F'$, there exists $s$ with $h_s<e_r$. With $s$ chosen
to be maximal with this property,
we have that $e_r$ lies in no
even subgraph of $\cE_s$, in contradiction of the properties of $F+F'$. 
Therefore, no such $e_r$ exists, so that $F+F'=\es$,
and $F =F'\in \ol\cF$ as required.
\end{proof}

Given any finitary basis $\cF=\set{C_1,C_2,\dots}$ of $\cE$, we may sample a
uniform random even subgraph of $G$ by extending
the recipe of Proposition 2.2 to infinite sums: we let
$\xi_1,\xi_2,\dots$ be independent
$\Be(\frac12)$ random variables and take
$\sum_i\xi_iC_i$. In other words, we take the sum of a
random subset of the finitary basis $\cF$ obtained by
selecting elements independently with probability $\frac12$ each.
Denote by $\rho$ the ensuing probability measure on $\cE$.

It turns out that $\rho$ is specified in a natural
way by its projections.
Let $E_1$ be a finite subset of $E$. The natural projection 
$\pei:\oi^{E}\to\oi^{E_1}$ given by $\pei(\om)=(\om_e)_{e\in E_1}$
maps $\sF$ onto a subspace $\sF_{E_1}=\pei(\sF)$ of $\oi^{E_1}$.

\begin{thm}\label{Tuniforminfinite}
  Let $G$ be a locally finite, 
infinite graph. The measure $\rho$ given above
is the unique probability measure on $\Omega=\oi^E$ such that, for
every finite set $E_1\subset E$ with $\sF_{E_1}\ne\es$,  
$(\om_e)_{e\in E_1}$ is uniformly distributed on $\sF_{E_1}$, i.e.,
\begin{equation}
  \label{eq:ax}
\rho(\pei^{-1}(A))=|A\cap \sF_{E_1}|/|\sF_{E_1}|,
\qquad
A\subseteq\oi^{E_1}. 
\end{equation}
\end{thm}

\begin{proof}
We may assume that $G$ is connected since, if not,
any $\rho$ satisfying \eqref{eq:ax} is a product measure over
the different components of $G$.
Note that every connected, locally finite graph is countable.

We show next that there is a unique probability measure
satisfying \eqref{eq:ax}.
This equation specifies its value on any cylinder event. By the
Kolmogorov extension theorem, it suffices to show that this
specification is consistent as $E_1$ varies, which amounts to showing
that if $E_1\subseteq E_2\subset E$ with $E_1,E_2$ finite, then the
projection $\pee:\oi^{E_2}\to\oi^{E_1}$ maps the uniform
distribution on $\sF_{E_2}$ to the uniform distribution on $\sF_{E_1}$. 
This is an immediate consequence of the
  fact that $\pee$ is a linear map of 
$\sF_{E_2}$ onto $\sF_{E_1}$. 

Finally we show that $\rho$ satisfies \eqref{eq:ax}.
Let $E_1 \subset E$ be finite. Since $\cF$ is finitary, its subset $\cF_1$,
containing cycles that intersect $E_1$, is finite. Since $\rho$ is
obtained from uniform product measure on $\cF$, its projection onto
$E_1$ is uniform (on its range) also.
\end{proof}

Diestel \cite[Chap.\ 8]{Dies} discusses related results for the space of subgraphs
spanned by the finite cycles, and relates them to closed curves in the
Freudenthal 
compactification of $G$ obtained by adding ends to the graph. It
is tempting to guess that there may be similar results for even
subgraphs and the one-point compactification of $G$ (where all
ends are identified to a single point at infinity). We do not
explore this here, except to note that the finite and infinite
cycles are exactly those subsets of the one-point compactification that
are homeomorphic to a circle. 

\section{Random even subgraphs via coupling}\label{sec:rcre}
We return 
to the random even subgraph with parameter $p\in[0,1)$ defined by
\eqref{eq:res} for a finite graph $G=(V,E)$.
We show next how to couple the $q=2$ \rc\ model and the random even subgraph of
$G$. Let $p\in[0,\frac12]$, and let $\om$ be a realization
of the \rc\ model on $G$ with parameters $2p$ and $q=2$.  
Let $R=(V,\g)$ be 
a uniform random even subgraph of $(V,\eta(\om))$.

\begin{thm}\label{thm:rceven}
Let $p\in [0,\frac12]$. The graph $R=(V,\g)$ is a random even
subgraph of $G$ with parameter $p$.
\end{thm}

This recipe for random even subgraphs provides a neat method for
their simulation, provided $p\le \frac12$. One may sample from the
\rc\ measure by 
the method of coupling from the past 
(see \cite{PW} and Section \ref{sec:sampling}), 
and then
sample a uniform random even subgraph by 
either Proposition \ref{P1} or Proposition \ref{P2}.

\begin{proof}
Let $g\subseteq E$ be 
even. By the observations
in Section~\ref{ssec:fin}, with $c(\om)=c(V,\eta(\om))$ denoting the
number of independent cycles in the open subgraph,
$$
\Pr(\g=g\mid \om) = 
\begin{cases} 2^{-c(\om)} &\text{if } g \subseteq \eta(\om),\\
0 &\text{otherwise},
\end{cases}
$$
so that
$$
\Pr(\g=g) = \sum_{\om: g\subseteq \eta(\om)} 2^{-c(\om)}\phi_{2p,2}(\om).
$$
Now $c(\om) = |\eta(\om)| - |V| + k(\om)$, so that, by \eqref{old2.53},
\begin{align*}
\Pr(\g=g) &\propto \sum_{\om: g\subseteq \eta(\om)}
(2p)^{|\eta(\om)|} (1-2p)^{|E\sm\eta(\om)|} 2^{k(\om)}
\frac 1{2^{|\eta(\om)| - |V| +k(\om)}}\\
&\propto \sum_{\om: g \subseteq \eta(\om)}
p^{|\eta(\om)|} (1-2p)^{|E\sm\eta(\om)|}\\
&=[p+(1-2p)]^{|E\sm g|} p^{|g|}\\
&= p^{|g|}(1-p)^{|E\sm g|},\qquad g \subseteq E.
\end{align*}
The claim follows.
\end{proof}

Let $p \in (\frac12,1)$. If $G$ is even, we can sample from $\rho_p$ by
first sampling a subgraph $(V,\tF)$ from $\rho_{1-p}$ and then taking
the complement $(V,E\setminus\tF)$, which has the distribution
$\rho_p$. If $G$ is not even, we adapt this recipe as follows.
For $W \subseteq V$ and $H \subseteq E$, we say that $H$ is
\emph{$W$-even} if each component of $(V,H)$ contains an even number of members of $W$.
Let $W\ne\es$ be the set of vertices of $G$ with odd degree, so that, in particular, $E$ is $W$-even.
Let $\Om^W = \{\om\in\Om: \eta(\om) \hbox{ is $W$-even}\}$. 
For $\om\in\Om^W$, we pick disjoint subsets $P^i=P^i_\om$, 
$i=1,2,\dots,\frac12 |W|$, of $\eta(\om)$, each of which
constitutes an open non-self-intersecting path with distinct endpoints lying in $W$,
and such that
every member of $W$ is the endpoint of exactly one such path.
Write $P_\om =\bigcup_i P^i_\om$.

Let $r=2(1-p)$, and let $\frW$ be the \rc\ measure on $\Om$ with parameters
$r$ and $q=2$ conditional on the event $\Om^W$. We sample from $\frW$ to obtain a subgraph $(V,\eta(\om))$,
from which we select a uniform random even subgraph $(V,\g)$ by the procedure of
the previous section. 

\begin{thm}\label{thm:rcodd}
Let $p \in (\frac12,1)$. The graph $S=(V,E\sm(\g\symdif P_\om))$ is a random even
subgraph of $G$ with parameter $p$.
\end{thm}

The recipes in Theorems \ref{thm:rceven} and \ref{thm:rcodd} 
can be combined as follows.  
Consider the generalized model
mentioned in Remark \ref{Rpe} with one parameter $p_e \in(0,1)$ for each edge
$e\in E$.
Let $A=\set{e\in E: p_e > \frac12}$. 
Define $r_e = 2p_e$ when $e \notin A$ and $r_e = 2(1-p_e)$  when $e\in A$. 
(Thus $0 < r_e \le 1$.)
Let $W = W_A$ be the set of vertices that are \emph{$A$-odd}, i.e., 
endpoints of an odd number of edges in $A$. 
Sample $\omega$  from the \rc\ measure with parameters
$\br = (r_e: e\in E)$ and $q=2$, conditioned on $\eta(\omega)$ being
$W$-even, let $P_\om$ be as above (for $W=W_A$),
and sample a uniform random even subgraph
$(V,\gamma)$ of $(V,\eta(\omega))$. 
For a discussion of relevant sampling techniques,
see Section \ref{sec:sampling}.

\begin{thm}\label{thm:rcmixed}
The graph $S=(V, \gamma \symdif P_\omega \symdif A)$
is a random even subgraph of $G$ with the distribution
$\rho_{\bp}$ given in \eqref{rpe}.
\end{thm}

Note that Theorems \ref{thm:rceven} and \ref{thm:rcodd} are 
special cases of Theorem \ref{thm:rcmixed}, with $A=\es$ and
$A=E$ respectively. 
We find it more illuminating to present the proof of Theorem
\ref{thm:rcodd} in this more general setup.

\begin{proof}[Proof of Theorem \ref{thm:rcmixed}, and thus of 
Theorem	\ref{thm:rcodd}] 

Let  $F = \gamma \symdif P_\omega \symdif A$ be the resulting edge-set,
and note that 
\begin{equation}\label{omfa}
   \eta(\omega) \supseteq  \gamma \symdif P_\omega = F \symdif A.
\end{equation}
Furthermore, if $F$ is even, then $F\symdif A$ has odd degree exactly
at vertices in $W=W_A$; 
hence \eqref{omfa} implies that necessarily  $\omega\in \Omega^W$.

Given an even edge-set $f\subseteq E$, we thus obtain $F=f$ if we first choose
$\omega\in\Omega^W$ with $\eta(\omega) \supseteq  f \symdif A$
and then (having chosen $P_\omega$) select $\gamma$ as the even subgraph
    $f \symdif A \symdif P_\omega$.
Hence, for every 
$\omega\in\Omega^W$ with    $\eta(\omega) \supseteq  f \symdif A$,
we have $\Pr(F=f\mid\om)=2^{-c(\om)}$, and summing over such $\om$ we find
\begin{equation*}
  \begin{split}
\Pr(F=f)
&\propto 
\sum_{\om:\eta(\om)\supseteq f\Delta A} 2^{-c(\om)} 
\phi_{\br,2}(\om)
\\&
\propto
\sum_{\om:\eta(\om)\supseteq f\symdif A} 
2^{-c(\omega)} 2^{k(\omega)} \prod_{e\in E} r_e^{\om(e)}(1-r_e)^{1-\om(e)} 
\\&
\propto
\sum_{\om:\eta(\om)\supseteq f\symdif A} 
2^{-|\eta(\omega)|}\prod_{e\in E} r_e^{\om(e)}(1-r_e)^{1-\om(e)} \\
&=
\sum_{\om:\eta(\om)\supseteq f\symdif A} 
\, \prod_{e\in E} \left(\frac{r_e}2\right)^{\om(e)}(1-r_e)^{1-\om(e)}  
\\&
=
\prod_{e\in f \symdif A}\left(\frac{r_e}{2}\right)  
\prod_{e\notin f \symdif A}\Bigl(1-\frac{r_e}{2} \Bigr).
  \end{split}
\end{equation*}
With $1_e$ denoting the indicator function of the event $\{e \in f\}$, 
this can be
rewritten as
\begin{equation*}
  \begin{split}
\Pr(F=f)
&\propto
\prod_{e\notin A} (r_e/2)^{1_e} (1-r_e/2)^{1-1_e} 
\prod_{e\in A}(r_e/2)^{1-1_e} (1-r_e/2)^{1_e} 
\\&
= \prod_{e\notin A} p_e^{1_e}
(1-p_e)^{1-1_e} \prod_{e\in A} (1-p_e)^{1-1_e} p_e^{1_e} \\
&= \prod_{e\in E} p_e^{1_e}
(1-p_e)^{1-1_e}
\\&
\propto 	\rho_{\bp}(f).	
  \end{split}
\end{equation*}
The claim follows.
\end{proof}

There is a converse to Theorem \ref{thm:rceven}. 
Take a random even subgraph $(V,F)$ 
of $G=(V,E)$ with parameter
$p\le \frac12$.  To each $e \notin F$, we assign an independent random colour,
\emph{blue} with probability $p/(1-p)$ and \emph{red} otherwise. Let $H$
be obtained from $F$ by adding in all blue edges.

\begin{thm}\label{thm:converse}
The graph $(V,H)$ has law $\phi_{2p,2}$.
\end{thm}

\begin{proof}
For $h \subseteq E$,
\begin{align*}
\Pr(H=h) &\propto \sum_{J\subseteq h,\ J\ \rm{even}}
\left(\frac p{1-p}\right)^{|J|} \left(\frac p{1-p}\right)^{|h \sm J|}
\left(\frac{1-2p}{1-p}\right)^{|E\sm h|}\\
&\propto p^{|h|}(1-2p)^{|E\sm h|} N(h),
\end{align*}
where $N(h)$ is the number of even subgraphs of $(V,h)$. As in
the above proof,
$N(h) = 2^{|h| - |V| + k(h)}$ where $k(h)$ is the number
of components of $(V,h)$, and the proof is complete.
\end{proof}

An edge $e$ of
a graph is called \emph{cyclic} if it belongs to some cycle
of the graph.

\begin{cor}\label{thm:cyclic}
For $p\in[0,\frac12]$ and $e \in E$, 
$$
\rho_p(\mbox{\rm $e$ is open}) = \tfrac12 \phi_{2p,2}(\mbox{\rm $e$ is
 a cyclic edge of the open graph}). 
$$
\end{cor}

By summing over $e \in E$, we deduce that the mean number of
open edges under $\rho_p$ is one half of the mean number of
cyclic edges under $\phi_{2p,2}$.

\begin{proof}
Let $\om \in \Om$ and let $\sC$ be a maximal
family of independent cycles of $\om$. 
Let $R=(V,\g)$ be a uniform random
even subgraph of $(V,\eta(\om))$, 
constructed using Proposition \ref{P1} and $\sC$.
For $e\in E$,
let $M_e$ be the number of elements of $\sC$ that include $e$. 
If $M_e \ge 1$, the number of these $M_e$ cycles of $\g$ that are
selected in the construction of $\g$
is equally likely
to be even as odd. Therefore,
$$
\Pr(e\in \g\mid\om)=\begin{cases} \tfrac12 &\mbox{if } M_e \ge 1,\\
0 &\mbox{if } M_e = 0.
\end{cases}
$$ 
The claim follows by Theorem \ref{thm:rceven}.
\end{proof}

\section{Sampling an even subgraph}\label{sec:sampling}

It was remarked earlier that Theorem \ref{thm:rceven} gives
a neat way of sampling an even subgraph of
$G$ according to the probability measure $\eta_p$ with $p \le \frac12$.
Simply use coupling-from-the-past (cftp) to sample from the \rc\ measure
$\phi_{2p,2}$, and then flip a fair coin once for each member of some
maximal independent set of cycles of $G$. 

The theory of cftp was enunciated
in \cite{PW} and has received much attention since. We recall that an
implementation of cftp runs for a random length of time $T$ whose tail
is bounded above by a geometric distribution; it terminates with
probability 1 with an exact sample from the target distribution. 
The \rc\ measure is one of the main examples
treated in \cite{PW}. We do not address questions of complexity and runtime 
in the current
paper, but we remind the reader of
the discussion in \cite{PW} of the relationship between the mean runtime
of cftp to that of the underlying Gibbs sampler.

The situation is slightly
more complicated when $p>\frac12$ and $G$ is not itself even, since the
conditioned \rc\ measure used in Theorems \ref{thm:rcodd} and \ref{thm:rcmixed} is neither
monotone nor anti-monotone. We indicate briefly in this section how to adapt the technique of cftp
to such a situation. 

Let $E$ be a non-empty finite set, and let $\mu$ be a probability measure
on the product space $\Om =\{0,1\}^E$. We call $\mu$ \emph{monotone} (\resp, \emph{anti-monotone}) if
$\mu(1_e\mid \xi_e)$ is non-decreasing (\resp, non-increasing) in $\xi\in\Om$.
Here, $1_e$ is the indicator function that $e$ is open, and
$\xi_e$ is the configuration obtained from $\xi$ on $E \sm \{e\}$.
For $e\in E$, $\psi\in\Om$, and $b=0,1$, we write $\psi_e^b$ for the configuration that agrees with $\psi$ off $e$ 
and takes the value $b$ on $e$. 

It is standard that cftp may be used
to sample from $\mu$ if $\mu$ is monotone (see \cite{PW,DW00}), and it is explained in
\cite{HagN} how to adapt this when $\mu$ is anti-monotone. We propose below a mechanism
that results in an exact sample from $\mu$ without any assumption of (anti-)monotonicity.
This mechanism may be cast in the more general framework of the `bounding chain' of \cite{Hub},
but, unlike in that work, it makes use of the fact that $\Om$ is partially ordered.
A similar approach was proposed in \cite{KM} under the title `dominated
CFTP', in the context of the simulation of point processes.

Write $S_\mu =\{\om\in\Om: \mu(\om)>0\}$, the subset of $\Om$ on which $\mu$ 
is strictly positive, and \emph{assume for simplicity that $S_\mu$
is increasing, and that $1 \in S_\mu$},
where $1$ (\resp, $0$) denotes the configuration of `all $1$'
(\resp, `all $0$'). 
This assumption is valid in the current setting,  
but is not necessary for all that follows.

We start with the usual Gibbs sampler for $\mu$. This is a discrete-time Markov chain $G=(G_n: n\ge 0)$ on
the state space $\Om$ that updates as follows. Suppose $G_n = \xi$. A uniformly
distributed member of $E$ is chosen, $e$ say, and also a random variable $U$ with the uniform
distribution on $[0,1]$. Then $G_{n+1} = \xi'$ where $\xi'(f) = \xi(f)$ for
$f \ne e$, and
$$
\xi'(e) = \begin{cases} 0 &\text{if } U > \mu(1_e\mid \xi_e),\\
1 &\text{if } U \le \mu(1_e\mid \xi_e).
\end{cases}
$$
The transition rule is well defined whenever $\xi_e^1\in S_\mu$. It is convenient
to use the device of \cite{HagN} to extend this definition to configurations not in $S_\mu$,
and to this end we set
\begin{equation}
\mu(1_e \mid \xi_e) = \max\bigl\{\mu(1_e\mid \psi_e): \psi_e \ge \xi_e,\ \psi_e^1\in S_\mu\bigr\}
\label{sc1}
\end{equation}
when $\xi_e^1 \notin S_\mu$. There is a degree of arbitrariness about this definition,
which we follow for consistency with \cite{HagN}.

Let $(e_n,U_n)$ be an independent sequence as above. Let $(A_n, B_n: n \ge 0)$
be a Markov chain with state space $\Om^2$, and $(A_0,B_0) = (0,1)$.
Suppose $(A_n,B_n) = (\xi,\eta)$ where $\xi \le \eta$. We set
$(A_{n+1},B_{n+1})= (\xi',\eta')$ where $\xi'(f)=\xi(f)$, $\eta'(f)=\eta(f)$
for $f \ne e_{n+1}$.  At $e=e_{n+1}$ we set 
\begin{align*}
\xi'(e) &= 1 \q\text{if and only if}\q U_{n+1} \le \alpha,\\
\eta'(e) &= 1 \q\text{if and only if}\q U_{n+1} \le \b,
\end{align*}
where
\begin{equation}
\begin{aligned}
\alpha=\alpha(\xi,\eta) &=\min\bigl\{\mu(1_e \mid \psi_e): \xi_e \le \psi_e \le \eta_e\bigr\},\\
\b=\b(\xi,\eta) &= \max\bigl\{\mu(1_e \mid \psi_e): \xi_e \le \psi_e \le \eta_e\bigr\}.
\end{aligned}
\label{sc2}
\end{equation}
Since $\alpha \le \b$, we have that $\xi'\le\eta'$.

We run the chain $(A,B)$ starting at negative times, in the manner prescribed
by cftp, and let 
$T$ be the coalescence time. More precisely, for $m \ge 0$, let
$(A_k(m),B_k(m): -m\le k \le 0)$ 
denote the chain beginning with $A_{-m}(m) = 0$, $B_{-m}(m)=1$, 
using a fixed random sequence $(e_n,U_n)_{-\infty}^0$ for all $m$, 
and set
$$
T=\min\{m\ge 0:  A_0(m)=B_0(m)\},
$$
so that $A_0(T) = B_0(T)$.

\begin{thm}\label{thm:cftp}
If $S_\mu$ is increasing and $1\in S_\mu$, then $P(T < \oo)=1$, and $A_0(T)$ has law $\mu$.
\end{thm}

\begin{proof}
We prove only that $P(T<\oo)=1$. 
The second part is a standard exercise in cftp, and is easily derived
as in \cite[Thm 2.2]{HagN}.
By the definition of $S_\mu$ and \eqref{sc1}, there exists $\eta = \eta(E,\mu) >0$ such that
$\mu(1_e \mid \xi_e) \ge \eta$ for all $e\in E$ and $\xi \in \Om$.
In any given time-interval
of length $|E|$, there is a strictly positive probability that the corresponding
sequence $(e_i,U_i)$ satisfies $E=\{e_i\}$ and $U_i < \eta$ for all $i$. On
this event,
the lower process $A$ takes the value 1 after the interval is past, so that coalescence
has taken place. The corresponding events for distinct time-intervals are independent, whence
the tail of $T$ is no greater than geometric. 
\end{proof}

The above recipe is exactly that of \cite{PW} when $\mu$
is monotone, and that of \cite{HagN} when $\mu$ is anti-monotone.

Let $G=(V,E)$ be a finite graph, and $W \subseteq V$ a non-empty set of vertices
with $|W|$ even.
Let $\br=(r_e: e \in E)$ be a vector of numbers from $(0,1]$, and let $\fbpq$ be the
\rc\ measure on $G$ with edge-parameters $\br$ and $q\ge 1$. We write $\fbpq^W$ for
$\fbpq$ conditioned on the event that the open graph is $W$-even, and note that $\fbpq^W$ is
neither monotone nor anti-monotone.
The event $S_\mu$ is
easily seen to be increasing, and $1 \in S_\mu$.
We may therefore
apply Theorem \ref{thm:cftp} to the measure $\mu=\fbpq^W$. 

Certain natural questions arise over the implementation of the above
algorithm, and we shall not investigate these here. 
First, it is convenient to have a quick way to
calculate $\alpha$ and $\b$ in \eqref{sc2}. A second
problem is to determine the mean runtime
of the algorithm, for which we remind the reader of
the arguments of \cite[Sect.\ 5]{PW}. 

\section{Random even subgraphs of planar lattices}\label{sec:revenL}
In this section, we consider random even subgraphs of the square and hexagonal lattices.
We show that properties of the Ising models on these lattices imply
properties of the random even graphs. In so doing, we shall review
certain known properties of the Ising model, and we include a `modern' proof
of the established fact that the Ising model on the square
lattice has a unique Gibbs state
at the critical point.

Let $G=(V,E)$ be a planar graph embedded in $\RR^2$, 
with dual graph $\Gd=(\Vd,\Ed)$, and write $e_\rd$ for
the dual edge corresponding to the primal edge $e\in E$. [See \cite{G-RC}
for an account of planar duality in the context of the \rc\ model.] Let $p\in(0,\frac12]$
and let $\om \in \Om = \{0,1\}^E$ have law $\phi_{2p,2}$. There is a one--one
correspondence between $\Om$ and $\Om_\rd=\{0,1\}^\Ed$ given by
$\om(e)+\om_\rd(e_\rd)=1$. It is well known that $\om_\rd$ has the law
of the \rc\ model on $\Gd$ with parameters
$(1-2p)/(1-p)$ and $2$, see \cite{G-RC} for example. 
For $A \subseteq V$, the \emph{boundary} of $A$ is given by
$\pd A = \{v\in A : v\sim w \text{ for some } w \notin A\}$. 
[A similar notation was used in a different context in \eqref{new1}.
Both usages are standard, and no confusion will arise in this section.]

For $\om\in\Om$, let $f_0,f_1,\dots,f_c$
be the faces of $(V,\eta(\om))$, with $f_0$ the infinite face. 
These faces are in one--one correspondence with the clusters of
$(\Vd,\eta(\om_\rd))$, which we thus denote by $K_0,K_1,\dots,K_c$, and
the boundaries of the finite faces form a basis of $\sF=\sF(V,\eta(\om))$.
More precisely, the boundary of each finite face $f_i$ consists of an
\lq outer boundary\rq{}  and zero, one or several \lq inner boundaries\rq; 
each of
these parts is a cycle (and two parts may have up to one vertex in common). If
we orient the outer boundary cycle counter-clockwise (positive) and the
inner boundary cycles clockwise (negative), the face will always be on the
left side along the boundaries, and the winding numbers of the 
boundary cycles sum up to 1 at every point inside the face and to 0
outside the face.
It is easy to see that the outer boundary cycles form a maximal family
of independent cycles of $(V,\eta(\om))$, and thus a basis of $\sF$;
another basis is obtained by the complete boundaries $C_i$ of the finite
faces.
We use the latter basis, and select a random subset of the basis by
randomly assigning (by fair coin tosses)
$+$ and $-$ to each cluster in the dual graph
$(\Vd,\eta(\om_\rd))$, or equivalently to each face $f_i$ of 
$(V,\eta(\om))$. We then select the boundaries $C_i$ of the finite faces $f_i$
that have been given a sign different from the sign of the infinite
face $f_0$. 
The union
(modulo 2) of the selected boundaries is by
Proposition~\ref{P1} and Theorem~\ref{thm:rceven} 
a random even subgraph of $G$ with parameter
$p$.
On the other hand, this union is
exactly the dual boundary of the
$+$ clusters 
of $\Gd$, that is, the set of open edges $e \in E$ with the 
property that one endpoint of the corresponding dual edge 
$\ed$ is labelled $+$ and the other is labelled $-$.
[Such an edge $\ed$ is called a $+/-$ edge.]

It is standard that the $+/-$ configuration on $\Gd$
is distributed as the Ising model on $\Gd$ with
parameter $\b$ satisfying
\be\label{eq:beta}
\frac{1-2p}{1-p} = 1 - e^{-2\b}
= \frac{2\tanh \b}{1+\tanh\b}.
\ee
In summary, we have the following.

\begin{thm}
  Let $G$ be a finite planar graph with dual $\Gd$.
A random even subgraph of $G$ with parameter
$p\in(0,\frac12]$ is dual 
to the $+/-$ edges of the Ising model on $\Gd$ with $\b$ satisfying
\eqref{eq:beta}. 
\end{thm}

Much is known about the Ising model on finite subsets of
two-dimensional lattices, and the above fact permits an analysis of
random even subgraphs of 
their dual lattices.
The situation is much more interesting in the infinite-volume limit, as follows.
Let $G=(V,E)$ be a finite subgraph of $\ZZ^2$, with boundary $\pd V$ when viewed
thus. A \emph{boundary condition} on $\pd V$ is a vector in 
$\Si_{\pd V} =\{-1,+1\}^{\pd V}$.
For given $\eta \in \Si_{\pd V}$, one may consider the 
Ising measure, denoted $\pi_V^\eta$, on $G$ conditioned
to agree with $\eta$ on $\pd V$. We call any subsequential weak limit of the family 
$\{\pi_V^\eta: \eta\in\Si_{\pd V},\, V \subseteq \ZZ^2\}$
a (weak limit) \emph{Gibbs state} for the Ising model.
It turns out that there exists a critical value of $\b$, denoted
$\bc$, such that there is a unique Gibbs state when $\b<\bc$, and
more than one Gibbs state when $\b>\bc$.

Consider the case when $G$ is a box in the square lattice $\ZZ^2$.
That is, $G=G_{m,n}$ is the subgraph of $\ZZ^2$ induced by
the vertex-set $[-m,m]\times[-n,n]$, where
$m,n\in\ZZ_+$ and $[a,b]$ is to be interpreted
as $[a,b]\cap \ZZ$. It is a mild inconvenience that $G_{m,n}$ is
not an even graph, and we adjust the `boundary' to rectify this.
For definiteness, we consider the so-called `wired boundary condition' on $G_{m,n}$,
which is to say that we consider the \rc\ measure on
the graph $\Gw_{m,n}$
 obtained from $G_{m,n}$ by identifying as one the set of vertices 
lying in its boundary $\pd G_{m,n}$.

It has been known since the work of
Onsager 
 that the Ising model on $\ZZ^2$ 
with parameter $\b$ is \emph{critical} when
$e^{2\b} = 1 + \sqrt 2$, or equivalently when the above \rc\ model
on the dual lattice has parameter satisfying
$$
\frac{1-2p}{1-p} = \frac{\sqrt 2}{1+\sqrt 2}
=2-\sqrt2,
$$
that is, $p=\pc$ where
\begin{equation}
  \pc=\frac 1{2+\sqrt 2}=1-\frac1{\sqrt2}.
\end{equation}
The Ising model has been studied extensively in the physics literature,
and physicists have a detailed knowledge of the two-dimensional
case particularly. There is a host of `exact calculations', rigorous
proofs of which can present challenges to mathematicians, see
\cite{Bax,McKW}. 

We shall use the established facts stated in the following theorem.
The continuity of the magnetization at the critical point contributes to
the proof that the re-scaled boundary of a large spin-cluster of the critical
Ising model converges weakly to the Schramm--L\"owner 
curve SLE$_3$, see \cite{Smir,Smi07}.

\begin{thm}
The critical value of the Ising model on the square lattice 
is $\bc = \bsd$ where $\bsd=\frac12\log(1 + \sqrt2)$ is the `self-dual point'. 
The magnetization (and therefore the corresponding \rc\
percolation-probability 
also) is a continuous function of $\b$ on $[0,\oo)$.
\end{thm}

We note the corollary that the wired and free \rc\ measures on $\ZZ^2$ 
are identical for $p\in[0,1]$; see \cite[Thms 5.33, 6.17]{G-RC}.

\begin{proof}
These facts are `classical' and have received much
attention, see \cite{McKW} for example; they may be proved as
follows using `modern' arguments. Recall first that the magnetization
equals the percolation probability of the corresponding wired \rc\ model,
and the  two-point correlation function of the Ising model
equals the two-point connectivity function of the \rc\ model (see \cite{G-RC}).
We have that $\bsd\le\bc$, by Theorem 6.17(a) of \cite{G-RC} or otherwise,
and similarly
the \rc\ model with free boundary condition
has percolation-probability 0 whenever either $\b\le\bsd$ or $\b<\bc$. 

By the results of
\cite{ABF,Si80}, the  two-point 
correlation function $\pi_\b(\s_x\s_y)$ of the spins at
$x$ and $y$ decays exponentially as $|x-y|\to\oo$ 
when $\b<\bc$, and it follows by the final statement of \cite{GG} or
otherwise that $\bc=\bsd$. 

The continuity of
the magnetization at $\b\ne\bc$ is standard, see for example
\cite[Thms 5.16, 6.17(b)]{G-RC}. When $\b=\bc$, it suffices to show that
the $\pm$ boundary-condition Gibbs states $\pi^{\pm}_\bc$ and the free
boundary-condition 
Gibbs state $\pi^0_\bc$ satisfy
$\pi^+_\bc=\pi^-_\bc=\pi^0_\bc$. Suppose this does not hold, 
so that $\pi^+_\bc \ne \pi^-_\bc\ne \pi^0_\bc$. By the \rc\
representation or otherwise, 
the two-point correlation functions
$\pi^\pm_\bc(\s_x\s_y)$ are
bounded away from 0
for all pairs $x$, $y$ of vertices. By the main result of \cite{Aiz80,Hig79}
(see also \cite{GeoHig}) and the symmetry of $\pi^0_\bc$, we have that
$\pi^0_\bc = \frac12 \pi^+_\bc + \frac12 \pi^-_\bc$, whence
$\pi^0_\bc(\s_x\s_y)$  
is bounded away from 0. By \cite[Thm 5.17]{G-RC}, this contradicts
the above remark 
that the percolation-probability of the free-boundary condition
\rc\ measure is 0 at $\b=\bsd=\bc$. 
\end{proof}

We consider now the so-called thermodynamic limit of the random even
graph on $\Gw_{m,n}$ as $m,n\to\oo$. 
It is long established 
that the (free boundary condition) Ising measure on $G_{m,n}$ 
converges weakly (in the product
topology) to an infinite-volume limit measure denoted $\pi_\b$. This
may be seen as follows using the theory of the corresponding \rc\ model
on $\ZZ^2$ (see \cite{G-RC}). When $\b\le\bc$, the existence
of the limit follows more or less as discussed above, using the coupling
with the \rc\ measure, and the fact that the percolation probability
of the latter measure is 0 whenever $\b\le\bc$. 
We write $\pi_\b$ for the limit Ising measure as $m,n\to\oo$.

The thermodynamic limit is slightly more subtle when $\b > \bc$, since the 
infinite-volume Ising model
has a multiplicity of Gibbs states in this case. 
The (wired) \rc\ measure on $\Gw_{m,n}$ converges to the wired limit measure.
By the uniqueness of infinite-volume \rc\ 
measures, the limit Ising measure is obtained by allocating random
spins to the clusters of the infinite-volume \rc\ model (see Section
4.6 of \cite{G-RC}).  Once again, we write $\pi_\b$ for the ensuing measure
on $\{-1,+1\}^{\ZZ^2}$, and we note that 
$\pi_\b = \frac12\pi^+_\b + \frac12 \pi^-_\b$ where $\pi^\pm_\b$ denotes
the infinite-volume Ising measure with $\pm$ boundary conditions.

It has been shown in \cite{CNPR} (see also \cite[Cor. 8.4]{GHM})
that there 
exists (with strictly positive $\pi_\b$-probability) an infinite spin-cluster
in the Ising model if and only if $\b>\bc$. More precisely:
\begin{letlist}
\item if $\b \le \bc$, there is $\pi_\b$-probability 1 that
all spin-clusters are finite,
\item if $\b > \bc$, there is $\pi_\b$-probability 1 that
there exists a unique infinite spin-cluster, which is
equally likely to be a $+$ cluster as a $-$ cluster. Furthermore,
by the main theorem of \cite{GKR} or otherwise, for any given finite set $S$ of
vertices, the infinite spin-cluster contains, $\pi_\b$-a.s., 
a cycle containing $S$ in its interior.
\end{letlist}
On passing to the dual graph, one finds that the random even subgraph
of $\Gw_{m,n}$
 with parameter 
$p\in (0,\frac12]$ converges weakly as $m,n \to\oo$ to a probability measure $\rho_p$
  that is concentrated 
on even subgraphs of $\ZZ^2$ and satisfies:
\begin{letlist}
\item[(\ax)] 
if $p\ge \pc$, there is $\rho_p$-probability 1
that all faces of the graph are bounded, 
\item[(\bx)]  
if $p < \pc$, there is $\rho_p$-probability 1
that the graph is the vertex-disjoint union of finite clusters.
\end{letlist}
(Note that \eqref{eq:beta} defines $\beta$ as a decreasing function of
$p$, so the order relations are reversed.) 

We have thus obtained a description of the weak-limit measure 
$\rho_p$ when $p\le \frac12$,
and we note the phase transition at the parameter-value 
$p=\pc$.
When $p>\frac12$, a random even subgraph of $\Gw_{m,n}$ is
the complement of a random even subgraph with parameter $1-p$.
[It is a convenience at this point that $\Gw_{m,n}$ is an even
 graph.]
Hence the weak-limit measure $\rho_p$ exists for all $p\in[0,1]$ and 
gives meaning to the expression
``a random even subgraph on $\ZZ^2$ with parameter
$p$''.  [It is easily verified that 
$\rho_{\frac12}$ 
equals the measure defined in Theorem \ref{Tuniforminfinite} for $\ZZ^2$, and
thus 
describes a uniform random even subgraph of $\ZZ^2$.]
There is a sense in which the random even subgraph on $\ZZ^2$
has two points of phase transition, corresponding to the values
$\pc$ and $1-\pc$.

We consider finally the question of the size of a typical face of
the random even graph on $\ZZ^2$ when $\pc \le p \le \frac12$.
When $p > \pc$, this amounts to asking about the size of a (sub)critical Ising spin-cluster.
Higuchi \cite{Higuchi} has proved an exponential upper bound
for the radius of such a cluster, and this has been extended
by van den Berg \cite{vdB08} to the cluster-volume.
Thus,
the law of the area of a typical face
has an exponential tail.
 
The picture is quite different when the square lattice is 
replaced by the hexagonal lattice $\HH$. Any even subgraph of $\HH$ has 
vertex degrees $0$ and/or $2$, and thus comprises a vertex-disjoint
union of cycles, doubly infinite paths, and isolated vertices.
The (dual) Ising model
inhabits the (Whitney) dual lattice of $\HH$, namely
the triangular lattice $\TT$. Once again there exists
a critical point $\pc=\pc(\TT)<\frac12$
such that the random even subgraph of $\HH$ satisfies
(\ax) and (\bx) above. 
In particular, the random even subgraph has a.s.\ only cycles and
isolated vertices but no infinite paths. 
Recall that site percolation on $\TT$ has critical
value $\frac12$. Therefore, 
for $p=\frac12$,  
the face $F_x$ of the random
even subgraph containing the dual vertex $x$ corresponds to
a \emph{critical} percolation cluster. It follows that its
volume and radius have polynomially decaying tails, and that
the boundary of $F_x$, when conditioned to be increasingly
large, approaches SLE$_6$. 
See \cite{Smi0,Smir} and \cite{CN,We07}. 
The spin-clusters of the Ising model on $\TT$ are `critical'
(in a certain sense described below)
for all $p\in(\pc(\TT),\frac12]$, and this suggests the possibility
that the boundary of $F_x$, when conditioned to be increasingly
large, approaches SLE$_6$ for any such $p$. This is supported
by the belief in the physics community that the so-called universality
class of the spin-clusters of the subcritical Ising model on $\TT$ is
the same as that of critical percolation. 

The `criticality' of
such Ising spin-clusters (mentioned above) may be obtained as follows.
Note first that, since $\b<\bc$, there is a unique Gibbs state $\pi_\b$ for
the Ising model. Therefore, $\pi_\b$ is
invariant under the interchange of spin-values $-1 \lra +1$.
Let $R_n$ be a rhombus of the lattice with side-lengths $n$ and axes
parallel
to the horizontal and one of the diagonal lattice directions, and
consider the event $A_n$ that $R_n$ is traversed from left to right by
a $+$ path 
(i.e., a path $\nu$  satisfying $\s_y= +1$ for
all $y \in \nu$). It
is easily seen that the complement of $A_n$ is 
the event that $R_n$ is crossed from top to bottom
by a $-$ path (see \cite[Lemma 11.21]{G99} for the 
analogous case of bond percolation on the square lattice). Therefore,
\begin{equation}
\pi_\b(A_n) = \tfrac12, \qq 0\le \b < \bc.
\label{eq:45}
\end{equation}
For $x\in \ZZ^2$,  let $S_x$ denote the spin-cluster containing $x$, and
define
$$
\rad(S_x) = \max\{|z-x|: z\in S\},
$$
where $|y|$ is the graph-theoretic distance from $0$ to $y$.
By \eqref{eq:45}, there exists a
vertex $x$ such that $\pi_\b(\rad(S_x) \ge n) \ge (2n)^{-1}$.
By the translation-invariance of $\pi_\b$, 
\be\label{lowerbnd}
\pi_\b\bigl(\rad(S_0)\ge n\bigr) \ge \frac 1{2n}, \qq 0\le \b <\bc,
\ee
where $0$ denotes the origin of the lattice.
The left side of \eqref{lowerbnd} tends to 0
as $n\to\oo$, and
the polynomial lower bound is an indication of the criticality of the 
model.

\section*{Acknowledgements} We acknowledge helpful discussions with
Michael Aizenman 
concerning the geometry of the Ising model with $\b<\bc$ on the square and
triangular lattices. Marek Biskup, Roman Koteck\'y,
and Senya Shlosman
kindly proposed the proof of the continuity at $\bc$ of the Ising
magnetization. 
We thank Jakob Bj\"ornberg, Bertrand Duplantier, Federico Camia,
Hans-Otto Georgii, Ben Graham, Alan Sokal, David Wilson,
and the referee
for their comments on parts of this work.
This research was mainly done during a visit by SJ to the University of
Cambridge, partly funded by Trinity College, and was completed during a visit
by GG to the University of British Columbia.
The authors acknowledge the hospitality
of the Isaac Newton Institute.

\def\PTRF{Probability Theory and Related Fields\ }
\def\CMP{Communications in Mathematical Physics\ }
\def\AoP{Annals of Probability\ }
\def\JPhysA{Journal of Physics A: Mathematical and General\ }
\def\JSP{Journal of Statistical Physics\ }

\bibliographystyle{amsplain}
\bibliography{even1}

\end{document}